\documentclass[11pt,reqno,a4paper]{article}

\usepackage[T1]{fontenc}
\usepackage[english]{babel}
\usepackage[utf8]{inputenc}
\usepackage{amsmath,amsthm,amssymb}
\usepackage{mathtools}
\usepackage{microtype}
\usepackage[a4paper,left=2.5cm,right=2.5cm,top=2.5cm,bottom=3cm,
  footskip=1.6cm]{geometry}

\usepackage{todonotes}

\numberwithin{equation}{section}

\newtheorem{theorem}{Theorem}[section]

\newtheorem{corollary}[theorem]{Corollary}
\newtheorem{proposition}[theorem]{Proposition}
\newtheorem{remark}[theorem]{Remark}

\usepackage[hidelinks]{hyperref}
\usepackage[nameinlink]{cleveref}

\newcommand{\lin}{\ensuremath{\mathrm{lin}}}

\newcommand{\C}{\mathbb C}

\newcommand{\ip}[2]{\langle #1,#2\rangle}
\newcommand{\norm}[1]{\lVert #1\rVert}
\DeclareMathOperator{\dist}{dist}
\DeclareMathOperator{\spann}{span}
\DeclareMathOperator{\ran}{ran}
\DeclareMathOperator{\tr}{tr}

\begin{document}

\title{A sharp bound for sampling widths in the uniform norm: \\kernel $D$-optimal designs and oversampling}

\author{Sebastian Neumayer \(\!{}^{a}\),
Tino Ullrich \(\!{}^{a,}\)\footnote{Corresponding author, Email:
tino.ullrich@math.tu-chemnitz.de}\\[2mm]
\({}^{a}\!\) Chemnitz University of Technology, Faculty of Mathematics}

\date{\today}
\maketitle

\begin{abstract}
For every reproducing kernel Hilbert space $\mathcal{H}_K$ with bounded kernel $K$, we prove that the linear sampling  and Gelfand widths in the uniform norm satisfy
\[
 g_m^{\lin}(B_{\mathcal H_K})_\infty
 \le \frac{m+1}{m-n+1}\,
 c_n(B_{\mathcal H_K})_\infty,
 \qquad m\ge n.
\]
In a certain sense, the leading factor \((m+1)/(m-n+1)\) on the right-hand side is sharp. We transfer the problem into the selection of a  maximal volume subset of kernel translates (kernel $D$-optimal design).
Apart from basic linear algebra, our proof employs low rank approximation techniques and a version of the Eckart-Young-Mirsky theorem. 
The same argument gives corresponding bounds for the selection  of a reduced basis in Hilbert spaces.
\end{abstract}

\section{Introduction}

Sampling recovery in the uniform norm has attracted significant interest in recent years \cite{PU,KKT21,KPUU23,KPUU24}, driven by several breakthroughs in the field. 
Note, however, that the problem has a much longer history, see \cite{KWW08, DuTeUl16}.  
In this note, we complement the recent results in \cite{NPU26} by a direct comparison between sampling and Gelfand widths.
While \cite{NPU26} focuses on constructive greedy methods for reproducing kernel Hilbert spaces (RKHS) $\mathcal{H}_K$, we prove an existence result based on kernel $D$-optimal designs. 

Let \(D\) be a nonempty set and let \(\mathcal H_K\) be a complex RKHS on \(D\), with reproducing kernel $K\colon D\times D \to \mathbb{C}$ satisfying
\begin{equation}
    \sup_{x\in D}\sqrt{K(x,x)}<\infty\quad \mbox{and}\quad f(x) = \langle f,K(\cdot,x)\rangle,\quad x\in D,
\end{equation}
so that \(\mathcal H_K\) embeds continuously into \(B(D)\), the space of bounded functions on \(D\).
As usual, let \(g_m^{\lin}(B_F)_\infty\) denote the linear sampling width of the unit ball $B_F$ of some function space $F$ that embeds continuously into \(B(D)\) \cite{NPU26}, i.e., 
\begin{equation}
g_m^{\lin}(B_F)_\infty \coloneqq \inf_{\substack{x_1,\dots,x_m\in D\\ \varphi\colon \C^m\to B(D) \text{ linear}}}
\sup_{\|f\|_F\leq 1} \bigl\| f-\varphi\bigl(f(x_1),\ldots,f(x_m)\bigr)\bigr\|_\infty \label{eq:glin},
\end{equation}
which represents the minimal worst-case error in the
uniform norm for linear recovery algorithms based on
\(m\) function values.
The corresponding worst-case error of arbitrary algorithms based on \(n\) bounded linear measurement functionals, namely the \(n\)th Gelfand width \(c_n\), is given by
\begin{equation}
c_n(B_F)_\infty
\coloneqq
\inf_{\substack{N\colon F\to\C^n\ {\rm bounded\ linear}\\\varphi\colon\C^n\to B(D)}} \sup_{\|f\|_F\leq 1}\norm{f-\varphi(Nf)}_\infty. \label{eq:cn}
\end{equation}
The recovery map \(\varphi\) in \eqref{eq:cn} is arbitrary.
In our setting (RKHS and uniform norm), the Gelfand widths equal the linear widths $a_n$, which are also known as approximation numbers.
Their precise definition is recalled in Section~\ref{sec:dictionary}. 
Our main result is a direct comparison of the two quantities for different indices $m$ and $n$ (oversampling) for RKHS with bounded kernel.
The result is based on maximizing the Gramian kernel determinant (kernel $D$-optimal design).
The term $D$-optimal design originates in statistics literature for optimal experimental designs.
However, the general principle (maximal volume concept) has been successfully applied at various places, see \cite{GoTy01, BaKaPoSchUl25} and the references therein. Our main result reads as follows.     

\begin{theorem}\label{thm:main}
Let \(\mathcal H_K\) be an RKHS on \(D\) with bounded kernel.
Then, for all integers \(m\ge n\ge1\),
\begin{equation}\label{eq:main}
g_m^{\lin}(B_{\mathcal H_K})_\infty
 \le\frac{m+1}{m-n+1}\,
 c_n(B_{\mathcal H_K})_\infty.
\end{equation}
Moreover, for every \(C<(m+1)/(m-n+1)\), there is an RKHS \(\mathcal H_K\) with strictly positive definite bounded kernel such that
\begin{equation}
  g_m^{\lin}(B_{\mathcal H_K})_\infty > C\,c_n(B_{\mathcal H_K})_\infty .
\end{equation}
\end{theorem}
Let us discuss several aspects and special cases of this result.
For $m=n$, we recover the well-known bound 
\begin{equation}
    g_n^{\lin}(B_{\mathcal H_K})_\infty
 \le(n+1) c_n(B_{\mathcal H_K})_\infty,
\end{equation}
coming as a consequence of \cite[Thm.\ 29.7]{NW3}.
In addition, Theorem \ref{thm:hilbert} recovers also \cite[Thm.~4.1]{BCDDPW11} in case $m=n$. 
With double oversampling, namely $m=2n$, we have 
\begin{equation}\label{f21}
    g_{2n}^{\lin}(B_{\mathcal H_K})_\infty
    \le\frac{2n+1}{n+1}\,c_n(B_{\mathcal H_K})_\infty\leq 2c_n(B_{\mathcal H_K})_\infty,
\end{equation}
which improves upon \cite[Thm.\ 20, (15)]{KPUU24} by a $\sqrt{n}$-factor for the case of an RKHS. Note that the analysis in \cite[Thm.\ 20, (15)]{KPUU24} also relies on a $D$-optimal design  for constructing the measure $\mu$ to reduce the problem to $L_2(\mu)$ recovery.

A direct precursor in the RKHS setting is \cite[Thm.~11]{KPUU23}, which establishes a bound of the form \eqref{f21} under additional structural and spectral assumptions on the kernel.
This result motivated both the
present work and \cite{NPU26}.
Theorem~\ref{thm:main} shows that \eqref{f21} holds under the minimal assumption of a bounded kernel.
It follows from the general Hilbert-space result in Theorem~\ref{thm:hilbert} through the transference principle of Section~\ref{sec:dictionary}, recently established in \cite{NPU26}.
Compared with \cite[Thm.~2.5]{NPU26}, it removes the logarithmic oversampling factor and any decay assumption on the widths.
The oversampling ratio \(b=m/n\) may be chosen arbitrarily close to \(1\), at the expense of a larger constant.
More precisely, for every \(b>1\), Theorem~\ref{thm:main} gives
\begin{equation}\label{eq:b}
g_{\lceil bn\rceil}^{\lin}(B_{\mathcal H_K})_\infty
\le \frac{b}{b-1}\,
c_n(B_{\mathcal H_K})_\infty,
\qquad n\ge1.
\end{equation}
For comparison, the \(L_2\)-recovery result \cite[Cor.~6.4]{BaSchUl23} has the larger factor
\((b/(b-1))^{3/2}\) multiplying the corresponding Kolmogorov width.
Its proof uses subsampling via spectral sparsification, whereas the present argument is based on \(D\)-optimal designs.

The paper is organized as follows.
Section~\ref{sec:selection} proves the Hilbert space subset selection bound Theorem~\ref{thm:hilbert}, which itself complements results from the reduced-basis community \cite{BuffaEtAl12, BCDDPW11, DPW13}, see also \cite{DPSW26, Wojtaszczyk15}.
In Section \ref{sec:sharpness}, we show that the leading factor on the right-hand side is optimal in a certain sense, see Proposition \ref{prop:sharp-index}. 
Afterwards, based on the recent results in \cite{NPU26}, we identify in Section \ref{sec:dictionary} sampling recovery in the uniform norm with a certain Hilbert space subset selection problem involving the kernel translates, where the error is measured in the $\mathcal{H}_K$-norm.
Similarly, the Gelfand widths are identified with corresponding Kolmogorov widths in $\mathcal{H}_K$. 
Putting all the pieces together finally enables us to prove Theorem \ref{thm:main} at the end of Section \ref{sec:dictionary}.

\paragraph{Notation.} Throughout, vectors and matrices are typeset in boldface.
Matrices $\mathbf A \in \C^{m \times n}$ are equipped with the spectral norm $\|\mathbf A\|_{2}$ and the Frobenius norm $\|\mathbf A\|^2_{F} = \sum_{i=1}^m\sum_{j=1}^n |a_{i,j}|^2$.
By $\mathbf{A}^*$, we denote the Hermitian (transpose) of $\mathbf{A}$ such that $\|\mathbf A\|^2_{F}  = \tr(\mathbf{A}^*\mathbf{A})$, where for a square matrix $\mathbf{B}=(b_{i,j})_{i,j} \in \C^{n\times n}$ the trace is defined as usual via $\tr(\mathbf{B}) = \sum_{i=1}^n b_{i,i} =  \sum_{i=1}^n \mathbf{u}_i^* \mathbf{B}\mathbf{u}_i$ for any orthonormal basis $\{\mathbf{u}_1,\ldots,\mathbf{u}_n\} \subset \C^n$.
Vectors $\mathbf{v} \in \C^n$ are frequently identified as column vectors (matrices of the form $\C^{n \times 1}$) and $\mathbf{v}^*$ as row vectors in $\C^{1\times n}$ with complex conjugated entries.
For a Hilbert space $\mathcal H$ and a closed subspace $V\subset \mathcal H$, we define $\dist_{\mathcal H}(a,V)\coloneqq  \inf_{g\in V}\|a-g\|_{\mathcal H}$.
In case of $\C^n$ equipped with the Euclidean norm, we also write $\mathcal H=\ell_2^n$.
\section{A Hilbert space subset selection bound}\label{sec:selection}

For a bounded subset \(A\) of a Hilbert space \(\mathcal H\), we define the \emph{subset width}
\begin{equation}\label{eq:tau_m}
    \tau_m(A)_{\mathcal H}
\coloneqq
\inf_{\substack{S\subset A\\|S|\le m}}
\sup_{a\in A}\dist_{\mathcal H}(a,\spann S).    
\end{equation}
In reduced-basis terminology, the selected elements are snapshots from
\(A\).
The following result controls the Hilbert space error for an optimal snapshot set in terms of the  Kolmogorov widths
\begin{equation}
    d_n(A)_{\mathcal H}
 \coloneqq
 \inf_{\substack{V\subset {\mathcal H}\\\dim V\le n}}
 \sup_{a\in A}\dist_{\mathcal H}(a,V).
\end{equation}
Note that \(\tau_m\) is an existence quantity: it measures the error of a best possible \(m\)-element subset, not the error of the set produced by a (weak) greedy algorithm as studied in \cite{BCDDPW11,BuffaEtAl12,DPW13,NPU26}.

\begin{theorem}[Hilbert subset selection]\label{thm:hilbert}
For every nonempty bounded subset \(A\) of a real or complex Hilbert
space \(\mathcal H\) and all integers \(m\ge n\ge1\), it holds that
\begin{equation}\label{eq:hilbert}
\tau_m(A)_{\mathcal H}
\le\frac{m+1}{m-n+1}\,d_n(A)_{\mathcal H}.
\end{equation}
\end{theorem}

\begin{proof}
We prove the result only for complex Hilbert spaces (in the case of Hilbert spaces over $\mathbb{R}$ the arguments are verbatim by replacing Hermitian $*$ by transpose $T$). 
Fix \(\eta>d_n( A)_{\mathcal H}\) and a subspace \(V\subset \mathcal H\),
\(\dim V\le n\), such that $\sup_{a\in A}\dist_{\mathcal H}(a,V)\le\eta$.
Now, we define
\begin{equation}
    \Delta_m\coloneqq \sup_{a_1,\ldots,a_m\in A} \det\bigl(\ip{a_j}{a_i}_{\mathcal H}\bigr)_{i,j=1}^m,
\end{equation}
which is finite due to the boundedness of $A \subset \mathcal H$ together with Hadamard's determinant inequality.
If \(\Delta_m=0\), no family of \(m\) elements from \( A\) is linearly independent.
Hence, \(\dim\spann A<m\), and choosing $S \subset A$ as a basis of \(\spann A\) in \eqref{eq:tau_m} gives \(\tau_m(A)_{\mathcal H}=0\) and \eqref{eq:hilbert} holds.
We may
therefore assume \(\Delta_m>0\).
Then, for any \(0<\varepsilon<1\), we can choose a linearly independent set \(S = \{a_1,\ldots,a_m\}\subset A\) (an \emph{\(\varepsilon\)-approximate
\(D\)-optimal design} for \(A\)) whose Gram matrix \(\mathbf G = (\ip{a_j}{a_i}_{\mathcal H})_{i,j=1}^m\) satisfies
\begin{equation}\label{f20}
\det\mathbf G\ge(1-\varepsilon)\Delta_m>0\,.
\end{equation}
In particular, \(\mathbf G\) is invertible.
Fix now \(a\in A\) with \(a\notin\spann S\) (otherwise $\dist_{\mathcal H}(a,\spann S) = 0$).
By defining $\mathbf g=(\ip{a}{a_i}_{\mathcal  H})_{i=1}^m \in \mathbb{C}^m$,
the Gram matrix $\mathbf G_a \in \mathbb C^{(m+1)\times (m+1)}$ of the extended set \(S_a = \{a_1,\ldots,a_m,a \}\subset A\) has the block form
\begin{equation}
    \mathbf G_a= \begin{pmatrix} \mathbf G&\mathbf g\\ \mathbf g^*&\norm{a}_{\mathcal H}^2
\end{pmatrix}.
\end{equation}
The matrix $\mathbf G_a$ is strictly positive definite because the elements of $S_a$ are linearly independent.
Now, the Schur complement identity \cite[0.8.5]{HoJo13} gives $\det\mathbf G_a =\det\mathbf G \cdot (\norm{a}_{\mathcal H}^2 -\mathbf g^*\mathbf G^{-1}\mathbf g)$.
Further, denoting with $P_{\spann S}\, a \coloneqq  \sum_{j=1}^m c_j a_j$ the projection onto $
\spann S$, we have  by \(a-P_{\spann S}\, a \perp \spann S\) that $\mathbf c$ satisfies \(\mathbf G \mathbf c = \mathbf g\).
Thus, we get
\begin{equation}
    \dist_{\mathcal H}(a,\spann S)^2
=\norm{a}_{\mathcal H}^2-\norm{P_{\spann S}\, a}_{\mathcal H}^2
=\norm{a}_{\mathcal H}^2-\mathbf g^*\mathbf G^{-1}\mathbf g,
\end{equation}
which in turn leads to 
\begin{equation}\label{eq:schur}
\det\mathbf G_a
 =\det\mathbf G \cdot \bigl(\norm{a}_{\mathcal H}^2 -\mathbf g^*\mathbf G^{-1}\mathbf g\bigr) = \det\mathbf G \cdot \dist_{\mathcal H}(a,\spann S)^2.
\end{equation}
For any $j = 1,\ldots,m+1$, we can define the \(m\)-submatrix  \(\mathbf G_a^{(j)} \in \C^{m\times m}\) of \(\mathbf G_a\in \C^{(m+1)\times (m+1)}\), where we delete the $j$th row and column, which is itself a Gram matrix for elements from $A$.
Thus, its determinant is non-negative and bounded by \(\Delta_m\).
Now, the adjugate formula and $\operatorname{adj}(\mathbf G_a)_{jj}
=\det\mathbf G_a^{(j)}$ imply that the trace of $\mathbf{G}_a^{-1}$ satisfies 
\begin{equation}\label{f15}
\tr(\mathbf G_a^{-1})\det \mathbf{G}_a = \tr\bigl(\operatorname{adj}(\mathbf G_a)\bigr)
=\sum_{j=1}^{m+1}\det\mathbf G_a^{(j)}
\leq (m+1)\Delta_m.
\end{equation}
Consequently, combining \eqref{f20}, \eqref{eq:schur}, and \eqref{f15} leads to
\begin{equation}
  \begin{split}
    \tr(\mathbf G_a^{-1})\det\mathbf G\cdot \dist_{\mathcal H}(a,\spann S)^2 = 
    \tr(\mathbf G_a^{-1})\det \mathbf{G}_a \leq
    (m+1)\Delta_m \leq \frac{m+1}{1-\varepsilon}\det \mathbf{G}.
  \end{split}
\end{equation}
This in turn gives
\begin{equation}\label{f17}
    \tr(\mathbf G_a^{-1})\dist_{\mathcal H}(a,\spann S)^2 \leq \frac{m+1}{1-\varepsilon}. 
\end{equation}
To establish \eqref{eq:hilbert}, it remains to estimate the trace on the left-hand side from below. 
Define the linear operator \(T\colon\C^{m+1}\to \mathcal H\) with $T \mathbf c
=\sum_{j=1}^{m}c_j a_j+c_{m+1}a$, so that we can make the identification \(\mathbf G_a=T^*T\). 
Denote with $\lambda_1 \geq \cdots \geq \lambda_{m+1}$ the non-increasing eigenvalues of \(\mathbf G_a=T^*T\). Below, we will prove the inequality  
\begin{equation}\label{f16}
    \sum\limits_{j=n+1}^{m+1} \lambda_j \leq \tr(T^*(I-P_V)T)=\sum\limits_{i=1}^{m} \dist_{\mathcal H}(a_i,V)^2 + \dist_{\mathcal H}(a,V)^2 \leq (m+1)\eta^2.
\end{equation}
Combining \eqref{f16}  with the inequality between the arithmetic and the harmonic mean leads to the required lower bound for $\tr(\mathbf{G}_a^{-1}) = \sum_{j=1}^{m+1}\lambda_j^{-1}$, namely 
\begin{equation}
(m-n+1)^{-1}\tr(\mathbf G_a^{-1})
\ge\frac{\sum_{j=n+1}^{m+1}\lambda_j^{-1}}{m-n+1} \geq \frac{m-n+1}{\sum_{j=n+1}^{m+1}\lambda_j}
\ge\frac{m-n+1}{(m+1)\eta^2}.
\end{equation}
Inserting this into \eqref{f17} shows that
\begin{equation}
  \dist_{\mathcal H}(a,\spann S)
\le\frac{m+1}{m-n+1}\,
\frac{\eta}{\sqrt{1-\varepsilon}}.  
\end{equation}
This estimate holds for every \(a\in A\), and the chosen set \(S\) is admissible in the definition of \(\tau_m\).
Taking the supremum over \(A\), and then letting \(\varepsilon\downarrow0\) and \(\eta\downarrow d_n( A)_{\mathcal H}\) proves \eqref{eq:hilbert}.

To establish \eqref{f16}, we denote $\mathbf{B} = T^*P_VT$ such that the first step is to estimate $\tr(\mathbf{G}_a-\mathbf{B})$ from below.
This is actually classical and a consequence of well-known results from low-rank approximation (Eckart-Young-Mirsky theorem \cite{EY36,Mi60}).
For the convenience of the reader, we give a self-contained proof.
Note that $\mathbf{G}_a-\mathbf{B}$ is positive semi-definite.
Since $\dim V \leq n$, $\mathbf{B}$ has rank at most $n$.
Thus, we may choose an orthonormal system $\{\mathbf{u}_1,\ldots,\mathbf{u}_k\}\subset \C^{m+1}$, $k=m+1-n$, in the kernel of $\mathbf{B}$ and set $U = \spann \{\mathbf{u}_1,\ldots,\mathbf{u}_k\}$.
Then, we have 
\begin{equation}\label{f18}
    \tr(\mathbf{G}_a-\mathbf{B}) \geq \sum\limits_{j=1}^k \bigl\langle (\mathbf{G}_a-\mathbf{B})\mathbf{u}_j,\mathbf{u}_j\bigr\rangle = \sum\limits_{j=1}^k \mathbf{u}_j^*\mathbf{G}_a\mathbf{u}_j  = \sum\limits_{j=1}^k \tr(\mathbf{G}_a\mathbf{u}_j\mathbf{u}_j^*) = \tr(\mathbf{G}_a\mathbf{P}_U),
\end{equation}
where we use the orthogonal projection $\mathbf{P}_U = \sum_{j=1}^k \mathbf{u}_j\mathbf{u}_j^*$.
Now, fix an orthonormal eigenbasis $\{\mathbf{v}_1,\ldots,\mathbf{v}_{m+1}\} \subset \C^{m+1}$ (corresponding to $\lambda_1,\ldots,\lambda_{m+1}$) of $\mathbf{G}_a$.
By construction, $p_i\coloneqq \|\mathbf{P}_U \mathbf{v}_i\|^2 \leq 1$ and $k = \tr(\mathbf{P}_U)= \sum_{i=1}^{m+1}p_i$.
Inserting this into \eqref{f18} yields
\begin{equation}\label{f19}
  \tr(\mathbf{G}_a-\mathbf{B}) \geq \tr(\mathbf{G}_a\mathbf{P}_U) = \sum\limits_{i=1}^{m+1}\langle \mathbf{G}_a \mathbf{P}_U\mathbf{v}_i, \mathbf{v}_i\rangle = \sum\limits_{i=1}^{m+1} \lambda_ip_i. 
  \end{equation}
Now, we observe that
\begin{equation}
\sum\limits_{i=1}^{m+1} \lambda_ip_i - \sum\limits_{i=n+1}^{m+1} \lambda_i = \sum\limits_{i=1}^{n} \lambda_ip_i - \sum\limits_{i=n+1}^{m+1} \lambda_i(1-p_i)\geq \lambda_n\Big(\sum\limits_{i=1}^{n} p_i - \sum\limits_{i=n+1}^{m+1} (1-p_i)\Big) = 0,
\end{equation}
which together with \eqref{f19} implies the first inequality in \eqref{f16}.
Regarding the remaining chain, we have
\begin{equation}
\begin{split}
\tr(T^*(I-P_V)T) &=  \sum\limits_{i=1}^{m+1} \langle (T^*(I-P_V)T)\mathbf{e}_i,\mathbf{e}_i\rangle_{\C^{m+1}} =\sum\limits_{i=1}^{m+1} \langle (I-P_V)T\mathbf{e}_i,T\mathbf{e}_i\rangle_{\mathcal H}\\
&=\sum\limits_{i=1}^{m} \langle (I-P_V)a_i,a_i\rangle_{\mathcal H} + \langle (I-P_V)a,a\rangle_{\mathcal H}\\
&=\sum\limits_{i=1}^{m} \dist_{\mathcal H}(a_i,V)^2 + \dist_{\mathcal H}(a,V)^2\leq (m+1)\eta^2.
\end{split}
\end{equation}
This concludes the proof.
\end{proof}
\begin{corollary}[Constant oversampling]\label{cor:rb}
Let \(A \subset \mathcal H \) be nonempty and bounded.
Then, for \(b>1\) and \(n\ge1\), it holds
\begin{equation}
    \tau_{\lceil bn\rceil}(A)_{\mathcal H}
\le\frac b{b-1}\,d_n(A)_{\mathcal H},
\qquad
\tau_{2n}( A)_{\mathcal H}
\le2\,d_n( A)_{\mathcal H}.    
\end{equation}
\end{corollary}

\begin{proof}
For \(m=\lceil bn\rceil\), the function
\(t\mapsto(t+1)/(t-n+1)\) is decreasing for $t > n - 1$, and hence
\begin{equation}
   \frac{m+1}{m-n+1}
\le\frac{bn+1}{(b-1)n+1}
\le\frac b{b-1}.
\end{equation}
The second estimate is the case $b=2$.
\end{proof}

\begin{remark}[Related literature]
{\em (i)} The quantity \(\tau_n\) is denoted by \(\bar d_n\) in \cite{Wojtaszczyk15} and \cite[Thm.~4.1]{BCDDPW11}, and called greedy Kolmogorov width in \cite[Sec.~1.3]{DPSW26}.
For the special case \(m=n\), \cite[Thm.~3.1]{Wojtaszczyk15} gives \(\tau_n\le(n+1)d_n\) for arbitrary compact sets in Banach spaces and shows that the order \(n\) is necessary; its Hilbert space case goes back to \cite[Thm.~4.1]{BCDDPW11}, see also \cite[Thm.~4.4]{DPSW26}.
The present contribution is the two-index estimate, which makes the benefit of oversampling explicit: constant oversampling replaces the same-index factor \(n+1\) by a dimension-independent constant. 
For Banach target norms, \cite[Thm.~2.6]{NPU26} obtains a factor of order \(\sqrt n\) under constant oversampling.

{\em (ii)} The restricted-width estimates in
\cite[Thms.~4.2 and~4.3]{DPSW26} either permit arbitrary subspaces of the closed linear span or require convexity and symmetry.
They therefore do not provide an oversampled comparison for arbitrary compact sets.

{\em (iii)} A related two-index expression occurs in Frobenius-norm column subset selection.
By \cite{GuSi12}, for \(m\ge n\) one can select \(m\) columns whose span approximates the best rank-\(n\) approximation within $\sqrt{(m+1)/(m-n+1)}$, and this trade-off is optimal up to lower-order terms.
The square root, compared with \eqref{eq:main}, reflects the use of an averaged Frobenius error rather than a uniform error.
The same distinction is visible at \(m=n\): \cite{KNU25} obtains the factor \(\sqrt{n+1}\) for \(f\in L^2(\Omega;\mathcal H)\), whereas \eqref{eq:main} and \cite[Thm.~4.1]{BCDDPW11} give \(n+1\) in the uniform setting.
Similar techniques for selecting submatrices of maximal volume \cite{GoTy01,GoOsSaTyZa10} and their randomized counterpart, volume sampling \cite{DeRaVeWa06}, are classical in low-rank matrix approximation.
\end{remark}

\section{Sharpness of the leading factor}\label{sec:sharpness}

We show that the constant in Theorem~\ref{thm:hilbert} cannot be replaced by any smaller one.
For each pair \(m\ge n\ge1\), we construct a one-parameter family \(A_\delta\subset\ell_2^{m+1}\) whose ratio \(\tau_m(A_\delta)/d_n(A_\delta)\) increases to \((m+1)/(m-n+1)\) as \(\delta\downarrow0\); the value itself is not attained.

\begin{proposition}\label{prop:sharp-index}
Let \(m\ge n\ge1\), and set \(N=m+1\) and \(k=m-n+1\).
For any \(0<\delta\le1\), there is a set
\(A_\delta=\{\mathbf a_1,\ldots,\mathbf a_N\}\subset\C^N\)
such that
\begin{align}
d_n( A_\delta)_{\ell_2^N}
&=\delta\sqrt{\frac kN}, \label{eq:sharp-d}\\
\tau_m(A_\delta)_{\ell_2^N}
&=\delta\sqrt{\frac{N}{k+n\delta^2}}. \label{eq:sharp-tau}
\end{align}
In particular, we get that 
\begin{equation}
    \frac{\tau_m( A_\delta)_{\ell_2^N}}
     {d_n( A_\delta)_{\ell_2^N}} =\frac{N}{\sqrt{k(k+n\delta^2)}} \longrightarrow \frac Nk=\frac{m+1}{m-n+1} \quad \text{as }\delta\downarrow0.
\end{equation}
\end{proposition}

\begin{proof}
Let $\mathbf F = N^{-1/2}(e^{2\pi\mathrm i(j-1)(\ell-1)/N})_{j,\ell=1}^N$ denote the discrete Fourier matrix, and let \(\mathbf P_U\) be the orthogonal
projection onto the span $U$ of its first \(k\) columns.
Let
\(\mathbf P_{U^\perp}=\mathbf I_N-\mathbf P_U\) denote the orthogonal projection onto the complement.
Since every entry of \(\mathbf F\) has absolute value \(N^{-1/2}\), \(\mathrm{rank}\, \mathbf P_U = k\), and \(\mathrm{rank}\, \mathbf P_{U^\perp} = n\),  we have for $1\le j\le N$ that
\begin{equation}\label{eq:constant-diagonal}
 \ip{\mathbf P_{U^\perp}\mathbf e_j}{\mathbf e_j}=\frac nN,
 \qquad
 \ip{\mathbf P_U\mathbf e_j}{\mathbf e_j}=\frac kN\,.
\end{equation}
Set \(\mathbf T_\delta=\mathbf  P_{U^\perp}+\delta\mathbf P_U\) and define
\begin{equation}\label{eq:lifted-vectors}
 \mathbf a_j\coloneqq \mathbf T_\delta\mathbf e_j.
\end{equation}
The operator \(\mathbf T_\delta\) is invertible, so the elements in \(A_\delta = \{\mathbf a_1, \ldots, \mathbf a_N\}\) are linearly independent.
Taking \(V=\ran\mathbf P_{U^\perp}\) gives
\begin{equation}
    \dist(\mathbf a_j,V) =\delta\norm{\mathbf P_U\mathbf e_j}_2=\delta\sqrt{\frac kN},
\end{equation}
\smash{and thus $d_n( A_\delta)_{\ell_2^N} \leq \delta\sqrt{k/N}$}.
For the converse, we use the Eckart-Young-Mirsky inequality \cite{EY36,Mi60} as for \eqref{f16}, which is here applied to $\mathbf{T}_\delta^*\mathbf{T}_\delta = \mathbf{P}_{U^\perp}+\delta^2\mathbf{P}_U$.
Fix \(W\subset\C^N\) with dimension at most \(n\).
The eigenvalues of \(\mathbf{T}_\delta^*\mathbf{T}_\delta\) are \(1\), with multiplicity \(n\), and \(\delta^2\), with multiplicity \(k=N-n\), so that \(\lambda_1=\dots=\lambda_n=1\) and \(\lambda_{n+1}=\dots=\lambda_N=\delta^2\) due to $\delta \leq 1$.
Thus, we get 
\begin{equation}
    \sum_{j=1}^N\dist(\mathbf a_j,W)^2 =\norm{\mathbf T_\delta-\mathbf P_W\mathbf T_\delta}_{F}^2 \ge \sum\limits_{i=n+1}^N \lambda_i = k\delta^2.
\end{equation}
At least one distance is therefore at least
\(\delta\sqrt{k/N}\), which proves \eqref{eq:sharp-d}.

Now, we turn to the second claim.
For \(\mathbf z_j=\mathbf T_\delta^{-*}\mathbf e_j\), we have \(\ip{\mathbf a_l}{\mathbf z_j}=\ip{\mathbf e_l}{\mathbf e_j}\).
Consequently, \(\mathbf z_j\) is orthogonal to $S_j \coloneqq \spann (A_\delta \setminus \{\mathbf a_j\})$, whereas \(\ip{\mathbf a_j}{\mathbf z_j}=1\).
The hyperplane-distance formula therefore gives
\begin{equation}\label{eq:omitted-distance}
 \dist_{\ell_2^N}(\mathbf a_j,S_j)
 =\norm{\mathbf z_j}^{-1}_2.
\end{equation}
Combining \(\mathbf T_\delta^{-1}
=\mathbf P_{U^\perp} + \delta^{-1}\mathbf P_U\) and \eqref{eq:constant-diagonal}
yields
\begin{equation}\label{eq:norm_z}
    \norm{\mathbf z_j}_2^2 =\bigl \langle (\mathbf P_{U^\perp} + \delta^{-2}\mathbf P_U)\mathbf e_j, \mathbf e_j \bigr \rangle = \frac{n\delta^2+k}{N\delta^2}.
\end{equation}
Every subset of $A_\delta$ with at most \(m=N-1\) vectors omits some \(\mathbf a_j\), so that \eqref{eq:omitted-distance} and \eqref{eq:norm_z} give the lower bound in
\eqref{eq:sharp-tau}. Selecting all vectors except one gives equality.
\end{proof}

\section{The kernel translates dictionary, sampling and Gelfand widths}\label{sec:dictionary}
Let \(\mathcal H_K\) be an RKHS on $D$ with kernel $K$.
In this section, we recall the tools and results from \cite{NPU26}.
These also involve the approximation numbers
\begin{equation}
    a_n(B_F)_\infty
\coloneqq \inf_{\substack{T\colon F\to B(D)\ {\rm bounded\ linear}\\ \operatorname{rank}T\le n}} \sup_{\|f\|_F\leq 1}\norm{f-Tf}_\infty.
\end{equation}
Further, we write $a_x\coloneqq K(\,\cdot\,,x)$ and $\mathcal K\coloneqq\{a_x:x\in D\}\subset\mathcal H_K$.
Given a finite set \(X\subset D\), let
\begin{equation}
    V_X\coloneqq\spann\{a_x:x\in X\},\qquad \operatorname{Pow}_X(x)
 \coloneqq\dist_{\mathcal H_K}(a_x,V_X).
\end{equation}
The associated subset width is
\begin{equation}\label{eq:tau-kernel}
\tau_m(\mathcal K)_{\mathcal H_K}
\coloneqq
\inf_{\substack{X\subset D\\|X|\le m}}
\sup_{x\in D}\operatorname{Pow}_X(x).
\end{equation}
The quantity $\mathrm{Pow}_X(\cdot)$ is the classical power function of kernel interpolation. In this sense, the $P$-greedy algorithm, whose convergence rates have been established in \cite{SaHaa17}, see also the stabilized and target data-dependent variants in \cite{WeSaHa21, WeSaHa23}, is precisely a greedy version of the kernel $D$-optimal design underlying Section~\ref{sec:selection}.
Accordingly, \cite{NPU26} analyzes the optimal recovery error produced by such greedy designs, whereas Theorem~\ref{thm:main} is an existence result for (approximate) maximizers.

Below, we recall the two key results concerning the relation of sampling and Gelfand widths. For the convenience of the reader, we include the proof ideas from \cite{NPU26}.
\begin{proposition}[Prop.\ 3.2 in \cite{NPU26}]\label{prop:sampling-dict}
For every \(m\ge0\), it holds that
\begin{equation}
    g_m^{\lin}(B_{\mathcal H_K})_\infty =\tau_m(\mathcal K)_{\mathcal H_K}.
\end{equation}
\end{proposition}
\begin{proof} 
{\em Step 1.} First, we establish $g_m^{\lin}(B_{\mathcal H_K})_\infty\le\tau_m(\mathcal{K})_{\mathcal{H}_K}$.
Fix a set $X=\{x_1,\dots,x_k\}\subset D$ with $k=|X|\le m$.
This is admissible in \eqref{eq:glin}, since repeating a node changes neither $V_X$ nor the resulting algorithm.
The well-known kernel interpolation, see \cite[Lem.\ 3.1]{NPU26}, is nothing other than the projection $f \mapsto P_{V_X}f$.
It is of course a linear algorithm and (due to RKHS structure) it uses only the samples $(f(x_1),\ldots,f(x_k))$ from $X$.
Since $f-P_{V_X}f\perp V_X$, we get for $f\in B_{\mathcal H_K}$ and $x\in D$ that
\begin{align}
    |f(x)-P_{V_X}f(x)| &= \bigl|\ip{f-P_{V_X}f}{a_x}_{\mathcal H_K}\bigr|
    = \bigl|\ip{f-P_{V_X}f}{(I-P_{V_X})a_x}_{\mathcal H_K}\bigr| \notag\\
    & \le \norm{f-P_{V_X}f}_{\mathcal H_K} \norm{ (I-P_{V_X})a_x  }_{\mathcal H_K}\le \norm{f}_{\mathcal H_K} \mathrm{Pow}_X(x) \le \mathrm{Pow}_X(x).
\end{align}
Taking the supremum over $x\in D$ and over $f\in B_{\mathcal H_K}$, it holds for every $X \subset D$ with $|X|\le m$ that
\begin{equation}
    g_m^{\lin}(B_{\mathcal H_K})_\infty \le
    \sup_{f\in B_{\mathcal H_K}}\norm{f-P_{V_X}f}_\infty \le \norm{\mathrm{Pow}_X}_\infty.
\end{equation}
Taking the infimum over all such $X$ gives $g_m^{\lin}(B_{\mathcal H_K})_\infty\le\tau_m(\mathcal K)_{\mathcal H_K}$.

{\em Step 2.} We now show $g_m^{\lin}(B_{\mathcal H_K})_\infty\ge\tau_m(\mathcal{K})_{\mathcal{H}_K}$.
To this end, we use a fooling argument, constructing two admissible functions that produce the same data but are far apart, so that no algorithm can reconstruct both equally well.
Let an arbitrary algorithm based on the samples $f(x_1),\dots,f(x_k)$ at locations $X=\{x_1,\dots,x_k\}$ with $k=|X|\le m$ be given, and let $L\in B(D)$ be its output on the data $(0,\dots,0)$.
If $\norm{\mathrm{Pow}_X}_\infty=0$ there is nothing to prove.
We fix $0<\delta<\norm{\mathrm{Pow}_X}_\infty$ and pick $x^*$ with $\mathrm{Pow}_X(x^*)\ge \norm{\mathrm{Pow}_X}_\infty-\delta>0$.
The function
\begin{equation}
    h \coloneqq \frac{(I-P_{V_X})a_{x^*}}{\mathrm{Pow}_X(x^*)}
\end{equation}
satisfies $\norm{h}_{\mathcal H_K}=1$, $h(x_i)=\ip{h}{a_{x_i}}_{\mathcal H_K}=0$ for all $i\le k$, and $h(x^*)=\ip{h}{a_{x^*}}_{\mathcal H_K}=\mathrm{Pow}_X(x^*)$.
Thus, both $h$ and $-h$ belong to $B_{\mathcal H_K}$ and produce the same data $h(x_i)=-h(x_i)=0$ for $x_i\in X$, so that the
algorithm returns the same $L$ in either case. By the triangle inequality, we have that
\begin{equation}
    \max\bigl\{\norm{h-L}_\infty, \norm{-h-L}_\infty\bigr\} \ge \norm{h}_\infty \ge |h(x^*)| = \mathrm{Pow}_X(x^*)\ge\norm{\mathrm{Pow}_X}_\infty-\delta,
\end{equation}
so that the worst-case error of the recovery algorithm is at least $\norm{\mathrm{Pow}_X}_\infty-\delta$.
Letting $\delta\to0$ and taking the infimum over $X$ gives $g_m^{\lin}(B_{\mathcal H_K})_\infty\ge\tau_m(\mathcal{K})_{\mathcal{H}_K}$.
\end{proof}

\begin{proposition}[Prop.\ 3.3 in \cite{NPU26}]
\label{prop:gelfand-dict}
For every \(n\ge0\), it holds that
\begin{equation}
    c_n(B_{\mathcal H_K})_\infty
 =a_n(B_{\mathcal H_K})_\infty
 =d_n(\mathcal K)_{\mathcal H_K}.
\end{equation}
\end{proposition}

\begin{proof}
We use the relation $c_n(B_{\mathcal H_K})_\infty= a_n(B_{\mathcal H_K})_\infty$ in the case $F = \mathcal{H}_{K}$, see \cite[Sect.\ 4.2.2]{NW1}.
Let $T$ be a linear operator of rank at most $n$.
We may assume $T = \sum_{j\le n}\ip{\cdot}{v_j}_{\mathcal H_K} g_j$ with $v_j\in \mathcal H_K$ and $g_j\in B(D)$, the functionals $f\mapsto\ip{f}{v_j}_{\mathcal H_K}$ being the Riesz representations of the bounded coefficient functionals of $T$.
Then for $f\in B_{\mathcal H_K}$ and $x\in D$, it holds $(f-Tf)(x) = \langle f, a_x-\sum_{j\le n}\overline{g_j(x)} v_j\rangle_{\mathcal H_K}$, so that
\begin{equation}
    \sup_{f\in B_{\mathcal H_K}}|(f-Tf)(x)| = \Bigl\| a_x-\sum_{j\le n}\overline{g_j(x)} v_j\Bigr\|_{\mathcal H_K}.
\end{equation}
Taking the supremum over $x$ shows that every linear operator $T$ of rank at most $n$ satisfies $\sup_{f\in B_{\mathcal H_K}}\norm{f-Tf}_\infty\ge\sup_{x\in D}\dist_{\mathcal H_K}(a_x,W)$ with $W=\spann\{v_1,\dots,v_n\}$, which in turn implies $a_n(B_{\mathcal H_K})_\infty\ge d_n(\mathcal{K})_{\mathcal H_K}$.

Conversely, fix $W\subset \mathcal H_K$ with $\dim W=l\le n$ and an orthonormal basis $w_1,\dots,w_l$ of $W$.
Then, let $T\coloneqq\sum_{j\le l}\ip{\cdot}{w_j}_{\mathcal H_K} w_j$, so that $Tf$ is the orthogonal projection of $f$ onto $W$ and $T$ has rank at most $n$. 
Its range functionals are $x\mapsto w_j(x)$, which by the reproducing property are bounded by $R\coloneqq \sup_{x\in D}\sqrt{K(x,x)}$ and hence admissible elements of $B(D)$. 
This gives $a_n(B_{\mathcal H_K})_\infty\le \sup_{x\in D}\dist_{\mathcal H_K}(a_x,W)$ and therefore $a_n(B_{\mathcal H_K})_\infty=d_n(\mathcal{K})_{\mathcal H_K}$.
\end{proof}

\begin{proof}[Proof of Theorem \ref{thm:main}]
We apply Theorem \ref{thm:hilbert} to \(A = \mathcal K\subset\mathcal H_K\).
The required boundedness of $A$ is a consequence of $\|K(\cdot,x)\|_{\mathcal H_K} = \sqrt{K(x,x)}\leq R$.
Then, we apply
Propositions~\ref{prop:sampling-dict} and~\ref{prop:gelfand-dict}. 

The sharpness assertion is derived from Proposition~\ref{prop:sharp-index} as follows. 
On \(D=\{1,\ldots,N\}\), define a kernel via $K_\delta(i,j)=\ip{\mathbf a_j}{\mathbf a_i}_{\ell_2^N}$, where the $\mathbf a_j \in \C^N$ are chosen as defined in \eqref{eq:lifted-vectors}.
This kernel is strictly positive definite (due to invertibility of $T_\delta$) and bounded since 
\begin{equation}
    K_\delta(j,j)=\norm{\mathbf a_j}_2^2 = \frac nN+\delta^2\frac kN\le1.
\end{equation}
Its kernel translates are isometric to \(A_\delta = \{\mathbf a_1,\ldots,\mathbf a_N\}\). Propositions
\ref{prop:sampling-dict} and~\ref{prop:gelfand-dict} now bridge from Proposition \ref{prop:sharp-index} to the sharpness in Theorem \ref{thm:main}. We have defined a sequence of RKHSs for which the leading factor is needed in the limit.
\end{proof}

\paragraph{Acknowledgment.}
TU is supported by the German Research Foundation (DFG) with grant Ul403/4-1.
The statement discovery and initial proof construction were assisted through LLMs by Anthropic and OpenAI.
Additional verification, presentation and writing were carefully carried out by the authors, who take full responsibility for this work.

\end{document}